\documentclass[11pt]{article}

\usepackage{amsfonts}
\usepackage{latexsym}
\usepackage{amsmath}
\usepackage{amssymb}
\usepackage{amsthm}
\usepackage{amsfonts, color}

\newtheorem{theorem}{Theorem}[section]
\newtheorem{corollary}[theorem]{Corollary}
\newtheorem{lemma}[theorem]{Lemma}
\newtheorem{proposition}[theorem]{Proposition}

\newtheorem{definition}[theorem]{Definition} 
\newtheorem{remark}[theorem]{Remark}

\numberwithin{equation}{section}

\newcommand{\IR}{{\mathbb R}}

\newcommand{\IE}{{\mathbb E}}

\newcommand{\inner}[2]{ \ensuremath { \langle {#1},{#2} \rangle } }

\begin{document}

\title{An inequality for moments of log-concave functions on Gaussian random vectors.}

\author{Nikos Dafnis\thanks{The author is supported by 
        the ERC Starting-Grant CONC-VIA-RIEMANN no. 637851.} 
        \ and 
        Grigoris Paouris\thanks{The author  is supported by 
        the US NSF grant CAREER-1151711 and BSF grant 2010288.}}
   
\date{}

\maketitle

\begin{abstract}
 We prove a sharp moment inequality for a log-concave or a log-convex function, 
 on Gaussian random vectors. As an application we take a stability result for the 
 classical logarithmic Sobolev inequality of L. Gross in the case where the function 
 is log-concave. 
\end{abstract}

\section{Introduction and main results}

\noindent A non-negative function $f:{\mathbb R}^k\rightarrow [0,+\infty)$
is called  \textit{$log$-concave on its support}, if and only if
\[
f\big((1-\lambda) x + \lambda y \big) \geq f(x)^{(1-\lambda)} f(y)^\lambda. 
\]
for every $\lambda \in [0,1]$ and $x, y \in {\rm supp}(f)$. Respectively, 
is called \textit{$\log$-convex on its support}, if nd only if  
\[ 
f\big((1-\lambda) x + \lambda y \big) \leq f(x)^{(1-\lambda)}f(y)^\lambda . 
\]
for every $\lambda \in [0,1]$ and $x, y \in {\rm supp}(f)$.
The aim of this note is to present a sharp inequality for Gaussian moments 
of a log-concave or a log-convex function, stated below as Theorem \ref{thm.sqrt-moments}. 

We work on ${\mathbb R}^k$, equipped with the standard scalar product $\langle\cdot ,\cdot\rangle$. 
We denote by $|\cdot|$, the corresponding Euclidean norm
and the absolute value of a real number. 
We additionally use the notation $X\sim N(\xi, T)$, if $X$ is a Gaussian random vector in 
${\mathbb R}^k$, with expectation $\xi\in{\mathbb R}^k$ and covariance the $k \times k$ 
positive semi-definite matrix $T$.
We say that $X$ is centered, whenever ${\mathbb E}X=0$, and that $X$ is a standard Gaussian 
random vector if it is centered with covariance matrix the identity in ${\mathbb R}^k$, where 
in that case $\gamma_k$ stands for its distribution law. Finally, ${\cal L}^{p,s}(\gamma_k)$ 
stand for the class of all functions $f\in L^p(\gamma_k)$ whose partial derivatives up to order s, 
are also in $L^p(\gamma_k)$. 


\begin{theorem}\label{thm.sqrt-moments}
Let $k\in{\mathbb N}$, $f:{\mathbb R}^k\rightarrow [0,+\infty)$ be a $\log$-concave, 
$g:{\mathbb R}^k\rightarrow [0,+\infty)$ be a $\log$-convex function,
and $X$ be Gaussian random vector in ${\mathbb R}^k$. Then,
\begin{itemize}
 \item[(i)]  for every  $r\in [0,1]$
             \begin{equation} \label{eq.sqrt-moments-1}
               {\mathbb E}f\big(\sqrt{r}X\big) \geq \left({\mathbb E}f(X)^r\right)^{\frac{1}{r}}
                \quad {\rm and} \quad
               {\mathbb E}g\big(\sqrt{r}X\big) \leq \left({\mathbb E}g(X)^r\right)^{\frac{1}{r}},
             \end{equation}
 \item[(ii)] for every $q\in[1,+\infty)$
             \begin{equation} \label{eq.sqrt-moments-2}
               {\mathbb E}f\big(\sqrt{q}X\big) \leq \left({\mathbb E}f(X)^q\right)^{\frac{1}{q}}
               \quad {\rm and} \quad
               {\mathbb E}g\big(\sqrt{q}X\big) \geq \left({\mathbb E}g(X)^q\right)^{\frac{1}{q}}.
             \end{equation}
\end{itemize}
In any case, equality holds if $r=1=q$ or if $f(x)= g(x)=e^{-\langle {\rm a}, x\rangle + c}$,
where ${\rm a}\in{\mathbb R}^k$ and $c\in{\mathbb R}$.
\end{theorem}

\smallskip

In section \ref{sec.2} we prove theorem \ref{thm.sqrt-moments}. In the main 
step of the proof, which is summarized in proposition \ref{prop.log-moments},  
we combine techniques from \cite{CDP} along with Barthe's inequality \cite{B}.

In section \ref{sec.entropy}, we prove a stability type result for the logarithmic 
Sobolev inequality. Let $X$ be a random vector in ${\mathbb R}^k$. Define the entropy 
of a function $f\in L(X)$, with respect to $X$, as
\[
{\rm Ent}_X(f) := {\mathbb E}|f(X)|\log |f(X)| - {\mathbb E}|f(X)|\, \log {\mathbb E}|f(X)|,
\]
provided that the expectations make sense. The Logarithmic Sobolev inequality, 
proved by L. Gross in \cite{Gross}, states that if $X \sim N(0,I_n)$, then 
\begin{equation} \label{eq.Log-Sob}
 {\rm Ent}_X(|f|^2) \leq 2\,{\mathbb E} |\nabla f(X)|^2 
\end{equation}
for every function $f\in L^2(\gamma_k)$. Of course we may state this for $f\geq 0$ without 
loss of generality. Moreover, Carlen proved in \cite{Carlen}, that equality holds if and only 
if $f$ is an exponential function.
For more details about the logarithmic Sobolev inequality we refer the 
reader to \cite{Bog}, \cite{LL}, \cite{Vill1}, \cite{Vill2} and to the references therein.

Theorem \ref{thm.sqrt-moments}, after an application of the Gaussian integration by parts formula
(see lemma \ref{IbP}), 
leads us to the following sharp, quantitative stability result for Gross' inequality, when the function 
is log concave.

\begin{theorem}\label{thm.Stab-Log-Sob}
 Let $X$ be a standard Gaussian random vector in ${\mathbb R}^k$ and $f=e^{-v}\in{\cal L}^{2,1}(\gamma_k)$,
 where $v:{\mathbb R}^k \rightarrow {\mathbb R}$ is a convex function (on its support). Then
 \begin{equation}\label{eq.quan-stab-L-S}
 2\,{\mathbb E}|\nabla f(X)|^2 - {\mathbb E}f(X)^2 \Delta v(X)
 \leq {\rm Ent}_X(f^2) \leq 2\,{\mathbb E} |\nabla f(X)|^2
\end{equation}
\end{theorem}

\medskip

\noindent {\bf Acknowledgement} Part of this work was done while the first named author was a 
postdoctoral research fellow at the Department of Mathematics at the University of Crete, and 
he was supported by the Action “Supporting Postdoctoral Researchers” of the Operational Program 
“Education and Lifelong Learning” (Action’s Beneficiary: General Secretariat for Research and 
Technology), co-financed by the European Social Fund (ESF) and the Greek State.

\section{Proof of the main result}\label{sec.2}

 The first main tool in the proof theorem \ref{thm.sqrt-moments} is the following inequality 
for Gaussian random vectors, proved in \cite{CDP}. Recall that for two $N\times N$ matrices $A$ 
and $B$, we say that $A\leq B$ if and only if $B-A$ is positive semi-definite.

\begin{theorem}\label{thm.CDP}
Let $m,n_1,\ldots,n_m\in{\mathbb N}$ and set $N=\sum_{i=1}^m n_i$. For every $1\leq i\leq m$, 
let $X_i$ be a Gaussian random vector in ${\mathbb R}^{n_i}$, such that ${\bf X}:=(X_1,\ldots,X_m)$,
is a Gaussian random vector in ${\mathbb R}^{N}$ with covariance the $N \times N$ matrix 
$T=(T_{ij})_{1\leq i,j\leq m}$, where $T_{ij}$ is the covariance matrix 
between $X_i$ and $X_j$ for $1\leq i,j\leq m$. Let $P$ be the block diagonal matrix,
\[
 P={\rm diag}(p_1T_{11},\ldots,p_mT_{mm}).
\]
Then for any set of nonnegative measurable 
functions $f_i$ on $\mathbb{R}^{n_i}$, $1\leq i \leq m$, 
 \begin{itemize}
  \item[$(i)$]  If $T\leq P$, then
              \begin{equation}\label{eq.CDP-up}
                \IE \prod_{i=1}^m f_i(X_i)\leq
                \prod_{i=1}^m \Big(\IE f_i(X_i)^{p_i}\Big)^{\frac{1}{p_i}}.
              \end{equation}
  \item[$(ii)$] If $T\geq P$, then
              \begin{equation}\label{eq.CDP-Down}
                \IE \prod_{i=1}^m f_i(X_i)\geq
                \prod_{i=1}^m \Big(\IE f_i(X_i)^{p_i}\Big)^{\frac{1}{p_i}}.
              \end{equation}
 \end{itemize}
\end{theorem}

\medskip 

 Theorem \ref{thm.CDP} generalizes many fundamental results in analysis, such as H\"{o}lder inequality
and its reverse, Sharp Young inequality and its reverse (see \cite{Bec} and \cite{BL}), and 
Nelson's Gaussian Hypercontractivity and its reverse (see \cite{Nel2} and \cite{MOS}). Actually,
the first part of theorem \ref{thm.CDP} is a reforlmulation of the famous Bascamp-Lieb inequality, 
first prooved in \cite{BL} (see also \cite{Lieb1} for the fully generalized version), while the 
second part provides us with its generalized reverse form. 

\medskip 

The second main tool in our proof, is the other famous reverse form of the Brascamp-Lieb inequality 
proved by F. Barthe \cite{B}, that generalizes the Pr\'{e}kopa-Leindler inequality. Next we state the 
Geometric form of Barthe's theorem, first put forward by k.Ball \cite{Ball1}:
 
\begin{theorem} \label{thm.Barthe}
 Let $n,m,n_1,\ldots,n_m \in{\mathbb N}$. For every $i=1,\ldots,m$ let $U_i$ be a 
 $n_i\times n$ matrix with $U_iU_i^*=I_{n_i}$ and $c_1,\ldots,c_m$ be positive numbers 
 such that
 \begin{equation*}
  \sum_{i=1}^m c_i\,U_i^*U_i = I_n
 \end{equation*}
 Let $h:{\mathbb R}^n \rightarrow [0,+\infty)$ and 
 $f_i:{\mathbb R}^{n_i} \rightarrow [0,+\infty)$, $i=1,\ldots,m$ measurable functions 
 such that
 \begin{equation}
  h\left(\sum_{i=1}^Nc_i U_i^*\xi_i\right) \geq 
  \prod_{i=1}^m f_i(\xi_i)^{c_i}\quad \forall\,\xi_i\in\IR^{n_i}
 \end{equation}
 then
 \begin{equation}
  \int_{{\mathbb R}^n} h(x)\,d\gamma_n(x) \geq 
  \prod_{i=1}^m \left(\int_{{\mathbb R}^{n_i}} f_i(x)\,d\gamma_{n_i}(x)\right)^{c_i}
 \end{equation}
\end{theorem}

\subsection{Decomposing the identity}

 We are going to apply theorem \ref{thm.CDP} in the special case where the covariance 
matrix is of the form $T= \big([T_{ij}]\big)_{i,j\leq n}$ $kn\times kn$, with $T_{ii}=I_k$ 
and $T_{ij} =tI_k$ if $i\neq j$, $t\in [-\frac{1}{n-1},1]$. Equivalently, in this case 
$X_{1}, \cdots, X_{n}$ are standard Gaussian random vectors in ${\mathbb R}^k$, such that 
\begin{equation}\label{eq.t-cor}
 {\mathbb E}(X_i X^*_j)=
 \left\{\begin{array}{rr} 
         I_k \,, & i=j \\
         tI_k \,, & i\neq j 
        \end{array}\right.
\end{equation}

For any $t\in[0,1]$, a natural way to construct such random vectors is to consider 
$n$ independent copies $Z_1,\ldots,Z_n$, of a $Z\sim N(0, I_k)$ and set 
\begin{equation*}\label{eq.sqrt{t}}
X_i:= \sqrt{t}\,Z + \sqrt{1-t}\,Z_{i}\,, \quad \ i=1,\ldots,n. 
\end{equation*}
It's then easy to check that condition \eqref{eq.t-cor} holds true for these vectors. 
However, we are going to construct such vectors using a more geometric language. We 
first make this construction the ``$k=1$" case of the theorems, and then we pass it 
for any $k\in{\mathbb N}$, using a tensorization argument. We begin with the definition
of the SR-simplex.

\medskip

\begin{definition}
 We say that $S={\rm conv}\{{\rm v}_1,\ldots,{\rm v}_n\} \subseteq {\mathbb R}^{n-1}$ 
 is the {\it spherico-regular simplex} (in short SR-simplex) if ${\rm v}_1,\ldots,{\rm v}_n$ 
 are unit vectors in ${\mathbb R}^{n-1}$ enjoying the properties 
 \begin{itemize}
  \item[(SR1)] $\langle {\rm v}_i,{\rm v}_j\rangle = -\frac{1}{n-1}$, for any $i\neq j$
  \item[(SR2)] $\sum_{i=1}^{n} {\rm v}_i=0$.
 \end{itemize}
\end{definition}
 
\noindent Using the vertices of the SR-simplex in ${\mathbb R}^{n-1}$, one can create $n$ vectors 
in ${\mathbb R}^n$ with the same angle between them. This is done in next lemma, which is a special
case of a more general fact, observed in \cite[sec. 3.1]{CDP}

\begin{lemma}\label{prop.decomp-t}
  Let $n\geq 2$ and ${\rm v}_1,\ldots,{\rm v}_n$ 
  be the vertices  of any RS-Simplex in $\IR^{n-1}$. For every 
  $t\in[-\frac{1}{n-1}\,,\,1]$, let $u_1,\ldots,u_n$ in ${\mathbb R}^n$ be the unit vectors
  in ${\mathbb R}^n$ with
  \begin{equation}
    u_i=u_i(t)=
    \sqrt{\frac{t(n-1)+1}{n}}\;\;e_n\;+\;\sqrt{\frac{n-1}{n}(1-t)}\;\;{\rm v}_i .
  \end{equation}
  Then we have that
  
  \begin{equation} \label{eq.eq-t}
    \inner{u_i}{u_j}=t \ ,\qquad \forall\,i\neq j .
  \end{equation}

  \medskip 

  \noindent Moreover, using those vectors we can decompose the identity in $\IR^n$:
  \begin{itemize}
   \item[(i)]  If $t\in[0,1]$, then
               \begin{equation}\label{eq.pos-dec}
	        \frac{1}{t(n-1)+1}\,\sum_{i=1}^n u_i u_i^* +
	        \frac{nt}{t(n-1)+1}\,\sum_{j=1}^{n-1} e_j e_j^* = I_n.
	       \end{equation}
   \item[(ii)] If $t\in[-\frac{1}{n-1},0]$, then
	       \begin{equation}\label{eq.neg-dec}
	        \frac{1}{1-t}\,\sum_{i=1}^n u_i u_i^* +
	        \frac{-nt}{1-t}\, e_n e_n^* = I_n.
	       \end{equation}
\end{itemize}
\end{lemma}
\begin{proof}
A direct computation shows that $\eqref{eq.eq-t}$, $\eqref{eq.pos-dec}$ 
and $\eqref{eq.neg-dec}$ holds true. 
\end{proof}

\medskip

\begin{remark}\label{rmk.rank-1}
 If $Z\sim N(0,I_n)$, then $X_i:=\langle u_i,Z\rangle$, $i=1,\ldots,n$, are standard 
 Gaussian random variables, satisfying the condition \eqref{eq.t-cor} in the $1$-dimensional 
 case. 
\end{remark}

In order to make the same construction in the general $k$-dimensional case, we use a 
more or less standard tensorization argument. We start with the definition of the 
\textit{tensor product} between two matrices.

\begin{definition}
 Let $A \in {\mathbb R}^{m \times n}$ and $B \in {\mathbb R}^{k \times \ell}$. Then the
 tensor product of $A$ and $B$ is the matrix
 \[
 A=\left[\begin{array}{ccc}
          a_{11}B & \cdots & a_{1n}B \\
           \vdots & \ddots & \vdots \\
          a_{m1}B & \cdots & a_{mn}B
         \end{array}
\right] \in {\mathbb R}^{km \times \ell n}.
\]
\end{definition}

Every vector $a\in{\mathbb R}^n$ is considered to be a column $n\times 1$ matrix, and with 
this notation in mind, we state some basic properties for the tensor product.

\begin{lemma}\label{lemma-tensor}
 \begin{enumerate}
  \item Let $a=(a_1,\ldots,a_m)^*\in{\mathbb R}^m$ and $b=(b_1,\ldots,b_n)^*\in{\mathbb R}^n$. 
        Then
        \[
         a\otimes b^*=ab^*=\left[\begin{array}{ccc}
                            a_1b_1 & \cdots & a_1b_n \\
                             \vdots & \ddots & \vdots \\
                            a_mb_1 & \cdots & a_mb_n
                            \end{array}\right] 
         \in {\mathbb R}^{m \times n}.
        \]
        As linear transformation: $a\otimes b^*=ab^*:{\mathbb R}^n\mapsto{\mathbb R}^m$ with
        \[
         (a\otimes b^*)(x)=(ab^*)(x)=\langle x,b \rangle\,a,
        \]
        for every $x\in{\mathbb R}^n$.
  \item Let $A_i \in {\mathbb R}^{m \times n}$ and $B \in {\mathbb R}^{k \times \ell}$. Then
        $
         \left(\sum_i A_i\right) \otimes B = \sum_i A_i \otimes B 
        $
  \item Let $A_1 \in {\mathbb R}^{m \times n}$, $B_1 \in {\mathbb R}^{k \times \ell}$,
        and $A_2 \in {\mathbb R}^{n \times r}$, $B_2 \in {\mathbb R}^{\ell \times s}$.Then
        \[
         \left(A_1 \otimes B_1\right) \left(A_2 \otimes B_2\right)
         = A_1A_2 \otimes B_1B_2 \in {\mathbb R}^{}
        \]
  \item For all $A$ and $B$,
        \[
         (A\otimes B)^* =A^* \otimes B^*
        \]
 \end{enumerate} 
\end{lemma}

\noindent Consider now the matrices

\begin{eqnarray}
 U_i&:=& u_i^* \otimes I_k = 
 \Big[
 \begin{array}{ccc}
        \big[u_{i1}I_k\big] & \cdots & \big[u_{in}I_k\big]
       \end{array} 
 \Big]\quad (k \times kn)\;, \qquad i=1,\ldots,n \label{eq.Ui} \\
 E_j&:=& e_j^* \otimes I_k = 
 \Big[
 \begin{array}{ccc}
        \big[e_{j1}I_k\big] & \cdots & \big[e_{jn}I_k\big]
       \end{array} 
 \Big]\quad (k \times kn)\;, \qquad j=1,\ldots,n. \label{eq.Ei}
\end{eqnarray}
Then,
\begin{align}
 U_i^*U_i &= (u_i^*\otimes I_k)^* (u_i^*\otimes I_k) = u_iu_i^*\otimes I_k,\quad kn \times kn \nonumber \\
 {\rm and} \qquad \quad & \nonumber \\
 E_j^*E_j &= (e_j^*\otimes I_k)^* (e_j^*\otimes I_k) = e_je_j^*\otimes I_k,\quad kn \times kn \nonumber
\end{align}
and thus, by taking the tensor product with $I_k$, in both sides of \eqref{eq.pos-dec}, we have that
\begin{equation}\label{eq.pos-dec-k}
  \frac{1}{p}\,\sum_{i=1}^n U_i^* U_i +
  \frac{nt}{p}\,\sum_{j=1}^{n-1} E_j^*E_j = I_{kn},
\end{equation}
for every $t\in[0,1]$, where $p:=(n-1)t+1$.

\medskip

With the help of these matrices we are ready now to construct the general situation, 
describing in \eqref{eq.t-cor}. We summarize in next lemma.

\begin{lemma}\label{lemma.multi-rank}
 Let $Z_1,\ldots,Z_n$ be iid $N(0,I_k)$, ${\bf Z}=(Z_1,\ldots,Z_n)\sim N(0,I_{kn})$, 
 end for every $i=1,\ldots,n$ consider the random vectors 
 \begin{equation}\label{eq.multi-rank}
  X_i:=U_i{\bf Z}=\sum_{a=1}^n u_{ia} Z_a
 \end{equation}
 Then $X_i \sim N(0,I_k)$ for every $i=1,\ldots,n$, while for $i\neq j$
 \begin{equation}\label{eq.t-cor-k}
  {\mathbb E} \big[X_i\otimes X_j^*\big]=\big[{\mathbb E}X_{ir}X_{j\ell}\big]_{r,\ell\leq k}
  =\big[t\delta_{r\ell}\big]_{r,\ell\leq k} = t I_k
 \end{equation}
\end{lemma}
\begin{proof}
 Clearly, ${\mathbb E}X_i = 0$, for every $i,j=1,\ldots,n$, and since
 \[
  {\mathbb E} \big[Z_a\otimes Z_b^*\big]
  =\big[{\mathbb E}Z_{ar}Z_{b\ell}\big]_{r,\ell\leq k}= \delta_{\alpha\beta}I_k
 \]
 we have that
 \begin{eqnarray*}
  {\mathbb E} X_{ir} X_{j\ell} 
  &=& {\mathbb E}\left(\sum_{a=1}^n u_{ia} Z_{ar}\right)\left(\sum_{b=1}^n u_{jb} Z_{b\ell}\right) \\
  &=& \sum_{a=1}^n\sum_{b=1}^n u_{ia} u_{jb}\,{\mathbb E} Z_{ar} Z_{b\ell} \\
  &=& \sum_{a=1}^n u_{ia} u_{ja}\,{\mathbb E} Z_{ar} Z_{a\ell} \\
  &=& \sum_{a=1}^n u_{ia} u_{ja}\,\delta_{r\ell} \\
  &=& \langle u_i, u_j \rangle \, \delta_{r\ell}.
 \end{eqnarray*}
 and from \eqref{eq.eq-t} the proof is complete.
\end{proof}

\subsection{Proof of theorem \ref{thm.sqrt-moments}}

Next proposition, that has a separate interest by its own, gives the first step for the proof 
of our main result, theorem \ref{thm.sqrt-moments}. 

\begin{proposition}\label{prop.log-moments}
 Let $t\in [0,1]$, $n\in{\mathbb N}$, $p = t(n-1)+1$, $X$ be a standard Gaussian random vector in 
 ${\mathbb R}^k$, $k\in{\mathbb N}$ and $X_1, \cdots, X_n$ be copies of $X$ such that 
 \[
  \mathbb E(X_i \, X_j^*)=\big({\mathbb E} X_{ir}X_{j\ell}\big)_{r,\ell\leq k}=tI_k, \quad i\neq j.
 \] 
 Then, for any $\log$-concave (on its support) function $f:{\mathbb R}^k\rightarrow [0,+\infty)$, 
 we have that
 \begin{equation} \label{eq.log-momment1}
 {\mathbb E} \left( \prod_{i=1}^{n} f(X_{i}) \right)^{\frac{1}{n}} \leq 
 \bigg( {\mathbb E} f(X)^{\frac{p}{n}}\bigg)^{\frac{n}{p}} \leq 
 {\mathbb E} f \left(\frac{1}{n} \sum_{i=1}^{n} X_i \right)
 \end{equation} 
 Note that, since $f$ is $\log$-concave we always have that 
 $\big(\prod_{i=1}^{n} f(X_{i})\big)^{\frac1n} \leq f\big(\frac{1}{n} \sum_{i=1}^{n} X_{i}\big)$, 
 while equality is achieved if 
 $f(x)= e^{\langle {\rm a},x \rangle + c}$, ${\rm a}\in{\mathbb R}^k$ and $c\in{\mathbb R}$. 
\end{proposition}
\begin{proof}
The left-hand side inequality in \eqref{eq.log-momment1}, follows after the application 
of theorem \ref{thm.CDP} in the special case describing in lemma \ref{lemma.multi-rank}.
Note that the assumption that $f$ is log-concave is not needed here. This inequality 
holds for any measurable function $f$. To make this more precise, the following simple 
remark is helpful.

\medskip

\begin{remark}
 Let $t\in[-\frac{1}{n-1},1]$ and $X_1,\ldots,X_n$ be standard Gaussian random vectors in 
 ${\mathbb R}^k$ satisfying the condition \eqref{eq.t-cor-k} of lemma \ref{lemma.multi-rank}. 
 Thus ${\bf X}:=(X_1,\ldots,X_n)$, is a centered Gaussian vector in ${\mathbb R}^{kn}$ with 
 covariance matrix $T=[T_{i,j}]_{i,j\leq n}$, with block entries the $k\times k$ matrices 
 $T_{ii}=I_k$ for every $i=1,\ldots,n$, and $T_{ij}=tI_k$, for $i\neq j$. If we set
 \[
  p:= (n-1)t+1\quad{\rm and}\quad q:=1-t,  
 \]
 then it's not hard to check that, for $t \leq 0$
 $q$ is the biggest and  $p$ is the smallest singular value of $T$. On the other hand, 
 if $t\geq 0$ then, $p$ is the biggest singular value of $T$ and $q$ is the smallest one. 
 Thus we have that
 \begin{itemize}
 \item[(i)]  if $t \geq 0$ then
             \[
              qI_{kn} \leq T \leq pI_{kn}
             \]
 \item[(ii)] if $t\leq 0$ then
             \[
              pI_{kn} \leq T \leq qI_{kn}
             \] 
\end{itemize}
\end{remark}

\noindent Thus, in the above situation, theorem \ref{thm.CDP} reads as follows: 
 
\begin{theorem}\label{thm.CDP-t}
 Let $k,n\in{\mathbb N}$, $t\in[-\frac{1}{n-1},1]$ and let $X_1,\ldots,X_n$ be 
 standard Gaussian random vectors in ${\mathbb R}^k$, with 
 ${\mathbb E} [X_i\otimes X_j^*]=tI_k$, for all $i\neq j$. Setting $p:=(n-1)t+1$ and $q:=1-t$,
 we have that for every set of measurable functions $f_i:{\mathbb R}^k \rightarrow [0,+\infty)$, 
 $i=1,\ldots,n$, 
 \begin{itemize}
  \item[(i)] if $t\in [0,1]$, then
        \begin{equation} \label{eq.t-pos}
          \prod_{i=1}^n\Big(\IE f_i(X_i)^{q}\Big)^{1/q}
          \leq\; \IE\prod_{i=1}^n f_i(X_i) \leq
          \prod_{i=1}^n\Big(\IE f_i(X_i)^p\Big)^{1/p},
        \end{equation}
  \item[(ii)] if $t\in [-\frac{1}{n-1}, 0]$, then
        \begin{equation} \label{eq.t-neg}
          \prod_{i=1}^n\Big(\IE f_i(X_i)^p\Big)^{1/p}
          \leq \IE\prod_{i=1}^n f_i(X_i) \leq
          \prod_{i=1}^n\Big(\IE f_i(X_i)^{q}\Big)^{1/q} 
        \end{equation}
 \end{itemize}
\end{theorem}

Now, the left-hand side inequality of \eqref{eq.log-momment1}, follows immediately from \eqref{eq.t-pos},
by taking $f_i=f$ for every $i=1,\ldots,n$.

\smallskip

In order to prove the right-hand side inequality of \eqref{eq.log-momment1}, we apply 
Barthe's theorem, using the decomposition of the identity \eqref{eq.pos-dec-k}. 
To do so we first state, in the following lemma, some technical details we are going to need. 

\begin{lemma}\label{lemma.tech-1}
 Let $U_i$ and $E_i$, $i=1,\ldots,n$ the matrices defined in 
 \eqref{eq.Ui} and \eqref{eq.Ei}, and set $p=(n-1)t+1$, $q=1-t$. Then
 \begin{eqnarray*}
  U_i^* &=& \sqrt{\frac{p}{n}}\,e_n\otimes I_k 
           + \sqrt{\frac{n-1}{n}q}\,{\rm v}_i\otimes I_k 
           \; \in {\mathbb R}^{kn \times k}. \\
  && \\
  U_i U_j^* &=& \langle u_i, u_j \rangle I_k \\
  && \\
  U_iE_j^* &=& \sqrt{\frac{n-1}{n}q}\,\langle {\rm v}_i, e_j \rangle I_k
 \end{eqnarray*}
 for every $i\leq n$ and $j\leq n-1$.
\end{lemma}
\begin{proof}
 The first and the second can be verified after some obvious and trivial computations. 
 For the third one, we have
 \begin{eqnarray*}
  U_iE_j^* 
  &=& (u_i^*\otimes I_k) (e_j^*\otimes I_k)^* \\
  &=& \left(\sqrt{\frac{p}{n}}\,e_n^*\otimes I_k 
     + \sqrt{\frac{n-1}{n}q}\,{\rm v}_i^*\otimes I_k\right) (e_j\otimes I_k) \\
  &=& \sqrt{\frac{p}{n}}\,(e_n^*\otimes I_k) (e_j\otimes I_k) 
     + \sqrt{\frac{n-1}{n}q}\,({\rm v}_i^*\otimes I_k) (e_j\otimes I_k) \\
  &=& \sqrt{\frac{p}{n}}\;e_n^*e_j\otimes I_k 
     + \sqrt{\frac{n-1}{n}q}\;{\rm v}_i^*e_j\otimes I_k \\
  &=& \sqrt{\frac{p}{n}}\;\langle e_n,e_j\rangle I_k 
     + \sqrt{\frac{n-1}{n}q}\;\langle {\rm v}_i,e_j\rangle I_k \\
  &=& {\mathbb O} + \sqrt{\frac{n-1}{n}q}\;\langle {\rm v}_i,e_j\rangle I_k.
 \end{eqnarray*}
\end{proof}

\medskip

\noindent To this end, we will apply Barthe's theorem \ref{thm.Barthe}, using the decomposition 
of the identity appearing in \eqref{eq.pos-dec-k}. More precisely, we choose the parameters: 
$n \leftrightarrow kn$, $m := 2n-1$, $n_i := k$ for all $i=1,\ldots,2n-1$, and
\[
 c_i:=\left\{\begin{array}{ccl}
              \frac{1}{p} & , & i=1,\ldots,n \\
             \frac{nt}{p} & , & i=n+1,\ldots,2n-1
            \end{array}\right. 
\]
and we apply theorem \ref{thm.Barthe} to the functions,
\[
 {\tilde f}_i(x):=\left\{\begin{array}{ccl}
                    f(x)^\frac{p}{n} & , & i=1,\ldots,n \\
                            1        & , & i=n+1,\ldots,2n-1
            \end{array}\right., \quad x\in{\mathbb R}^k
\]
and
\[
 h(x):= f\left(\frac1n\sum_{i=1}^nU_ix\right),\quad x\in{\mathbb R}^{kn}.
\]

\noindent Note then that under lemma \ref{lemma.tech-1}, we have that
for every $\xi_j\in{\mathbb R}^k$, $j=1,\ldots,n$,
\begin{eqnarray*}
  && h\left(\sum_{j=1}^n\frac1pU_j^*\xi_j + \sum_{a=1}^{n-1}\frac{nt}{p}E_a^*\xi_{n+a}\right)  \\
  &=& f\left(\frac1n\sum_{i=1}^n \sum_{j=1}^n \frac1p U_i U_j^* \xi_j 
      + \frac1n\sum_{i=1}^n \sum_{a=1}^{n-1} \frac{nt}{p} U_i E_a^* \xi_{n+a}\right) \\
  &=& f\left(\frac1n\sum_{i=1}^n \sum_{j=1}^n \frac1p U_i U_j^* \xi_j 
      + \frac1n\sum_{i=1}^n \sum_{a=1}^{n-1} 
        \frac{nt}{p} \sqrt{\frac{n-1}{n}q}\langle {\rm v}_i ,e_a \rangle \xi_{n+a}\right) \\ 
  &=& f\left(\frac1n\sum_{i=1}^n \sum_{j=1}^n \frac1p U_i U_j^* \xi_j \right) 
  \qquad \left({\rm since} \;\; \sum {\rm v}_i=0 \right) \\ 
  &=& f\left(\frac1n\sum_{i=1}^n \sum_{j=1}^n \frac1p \langle u_i , u_j \rangle \xi_j \right) \\
  &=& f\left(\frac1n\sum_{i=1}^n \Big(\frac1p \xi_i + \sum_{j\neq i} \frac{t}{p} \xi_j \Big) \right) \\ 
  &=& f\left(\frac1n\sum_{i=1}^n \Big(\frac1p + (n-1) \frac{t}{p} \Big) \xi_i \right) \\ 
  &=& f\left(\frac1n\sum_{i=1}^n \xi_i \right) 
  \geq \prod_{i=1}^n f(\xi_i)^\frac{1}{n} = \prod_{i=1}^n \Big(f(\xi_i)^\frac{p}{n}\Big)^\frac1p
  = \prod_{i=1}^n {\tilde f}(\xi_i)^{c_i}
\end{eqnarray*}
Thus, theorem \ref{thm.Barthe} gives that
\begin{equation}\label{eq.*}
  {\mathbb E} f\left(\frac1n\sum_{i=1}^n X_i\right) 
  = {\mathbb E} f\left(\frac1n\sum_{i=1}^n U_iZ \right)
  \geq \prod_{i=1}^n\left({\mathbb E}f(X_i)^\frac{p}{n}\right)^\frac{1}{p}
  =\left({\mathbb E}f(X)^\frac{p}{n}\right)^\frac{n}{p} 
\end{equation}
and the proof is complete
\end{proof}


\begin{proof}[Proof of theorem \ref{thm.sqrt-moments}.]
Suppose first that $X \sim N(0,I_k)$. Then, under the notation of lemma \ref{lemma.multi-rank} 
we have that
\begin{align*}
 \frac1n\sum_{i=1}^n U_i {\bf Z} 
 &= \frac1n\sum_{i=1}^n \sqrt{\frac{p}{n}}\,(e_n^*\otimes I_k){\bf Z} \;+\; 
    \frac1n\sum_{i=1}^n \sqrt{\frac{n-1}{n}q}\,({\rm v}_i^*\otimes I_k){\bf Z} \\
 &= \sqrt{\frac{p}{n}}\,(e_n^*\otimes I_k){\bf Z} \;+\; 
    \frac1n \sqrt{\frac{n-1}{n}q}\,\left(\sum_{i=1}^n {\rm v}_i^*\right)\otimes I_k\,{\bf Z} \\
 &= \sqrt{\frac{p}{n}}\,E_n {\bf Z} \;+\; 
    \frac1n \sqrt{\frac{n-1}{n}q}\,\left(\sum_{i=1}^n {\rm v}_i\right)^*\otimes I_k\,{\bf Z} \\
 &= \sqrt{\frac{p}{n}}\,Z_n.
\end{align*}
Thus, the right hand side of \eqref{eq.log-momment1} can be written as 
\begin{equation}\label{eq.sqrt-p/n}
 {\mathbb E}f\left(\sqrt{\frac{p}{n}}\,X\right)
 \geq \left(f(X)^\frac{p}{n}\right)^\frac{n}{p}.
\end{equation}
where $p=(n-1)t+1$, $n\in{\mathbb N}$, and $t\in [0,1]$.  

\medskip

Consequently, if $f:{\mathbb R}^k\rightarrow [0,+\infty)$ is a log-concave function 
and $r\in(0,1]$, then there exist, $t\in[0,1]$ and $n\in{\mathbb N}$, such that 
$r=\frac{p}{n}=\frac{(n-1)t+1}{n}$, and so by \eqref{eq.sqrt-p/n} we get that
\begin{equation}\label{eq.moment3}
 {\mathbb E} f\big(\sqrt{r}X\big) \geq \left({\mathbb E} f(X)^r\right)^\frac1r
\end{equation}
for every $r\in(0,1]$. We deal independently with the case where $r=0$. Since $f$ is 
$log$-concave, there exists a convex function $v:{\mathbb R}^k\mapsto{\mathbb R}$, such that 
$f= e^{-v}$. Then for $r=0$, inequality \eqref{eq.sqrt-moments-1} is equivalent to Jensen's
inequality
\begin{equation}\label{eq.moment4}
 v(0)=v({\mathbb E}X) \leq {\mathbb E}v(X),
\end{equation} 
and the proof of \eqref{eq.sqrt-moments-1} is now complete.

\medskip

For every $q \geq 1$ consider $r=\frac{1}{q}\in(0,1]$. Let 
 $F(x)=f(x/\sqrt{r})^{1/r}$ which is also $\log$-concave and so \eqref{eq.moment3} 
 for $F$ and $r$ implies
 \begin{equation} \label{eq.moment.5}
  {\mathbb E} f( X)^q \geq \big({\mathbb E} f(\sqrt{q} X)\big)^q,
 \end{equation}
and \eqref{eq.sqrt-moments-2} follows.

\medskip

\noindent Assume now that $g:\mathbb R^{n}\rightarrow[0,+\infty)$ is $\log$-convex and $r\in(0,1]$.
By the log-convexity of $g$ and theorem \ref{thm.CDP-t}(i), we have that 
\begin{equation}\label{eq.moment.6}
 \mathbb E g\left(\frac{1}{n}\sum_{i=1}^{n}X_{i}\right) \leq 
 \mathbb E \prod_{i=1}^{n}g(X_{i})^{\frac{1}{n}}  \leq 
 \left( \mathbb E g(Z)^{\frac{p}{n}} \right)^{\frac{n}{p}} .
\end{equation}
As we have seen at the beginning of the proof, we have that 
$\frac1n\sum_{i=1}^n X_i  \sim \sqrt{\frac{p}{n}}\,X$. So, using \eqref{eq.moment.6} 
for $t\in[0,1]$ and $n\in{\mathbb N}$ such that $\frac{p}{n}=\frac{(n-1)t+1}{n}=r$, we derive that 
\[ 
 \mathbb E g\left( \sqrt{r} Z\right) \leq  \left( \mathbb  Eg(Z)^r \right)^{\frac1r}. 
\]
for every $r\in (0,1]$. The rest of the proof for a $\log$-convex function $g$ 
is identical to the $\log$-concave one. 

Finally for the equality case, 
a straightforward computation shows that for $f(x)= e^{\langle {\rm a},x\rangle+c}$, we have that
\[
 {\mathbb E} f(\sqrt{q}X)= 
 C \exp\left(\frac{q}{2}|{\rm a}|^2\right)
 =\big({\mathbb E}f(X)^q\big)^\frac{1}{q}. 
\]
for every $q \geq 0$.

\medskip

\noindent At the end, suppose that $X$ is a general Gaussian random vector in ${\mathbb R}^k$ with 
expectation $\xi\in{\mathbb R}^k$ and covariance matrix $T=UU^*$ where $U\in{\mathbb R}^{k\times k}$. 
Note, that if $f$ is a log-concave (or log-convex) and positive function on ${\mathbb R}^k$, then so is 
$F(x):=f(Ux-\xi)$. Moreover, if $Z\sim N(0,I_k)$ then $UZ-\xi \stackrel{d}{=} X \sim N(0,T)$. Thus, 
we get the general theorem by applying the previous case with function $F$.

\end{proof}

\section {Entropy Inequalities - Stability in Log-Sobolev} \label{sec.entropy}

\begin{proposition}\label{cor-main1}
 Let $X$ be a Gaussian random vector in ${\mathbb R}^k$, and $f:{\mathbb R}^k\rightarrow [0,+\infty)$. 
 Then,
 \begin{itemize}
  \item [(i)]  if $f$ is log-concave, then
               \begin{equation}\label{eq.cor-main-1}
                {\rm Ent}_X(f) \geq \frac12{\mathbb E}\langle X ,\nabla f(X) \rangle 
               \end{equation}
  \item [(ii)] if $f$ is log-convex, then
               \begin{equation}\label{eq.cor-main-1b}
                {\rm Ent}_X(f) \leq \frac12{\mathbb E}\langle X ,\nabla f(X) \rangle 
               \end{equation}
 \end{itemize}
In any case, one has equality when $f(x)= \exp\big(\langle {\rm a},x\rangle+c\big)$, 
${\rm a}\in{\mathbb R}^k$, $c\in{\mathbb R}$. 
\end{proposition}
\begin{proof}
Let $M(q):= \big({\mathbb E}f(X)^q \big)^\frac1q$ and $H(q):= {\mathbb E}f(\sqrt{q}X)$. 
Then we have that 
\[
 M(1) = {\mathbb E}f(X) = H(1),
 \;\;
 M^\prime(1) = {\rm Ent}_X(f)
 \;\; {\rm and} \;\;
 H^\prime(1) = \frac12 {\mathbb E} \langle X,\nabla f(X)\rangle.
\]
Thus, Theorem \ref{thm.sqrt-moments} immediately implies the desired result.
\end{proof}

\medskip


\noindent Gaussian random vectors have a special property: the  
\textit{Gaussian Integration by Parts} formula, which we state in the next lemma 
(see \cite[Appendix 4]{Tal1} for a simple proof). 

\begin{lemma}\label{IbP}
 Let $X,Y_1,\ldots,Y_n$ be centered jointly Gaussian random variables, and $F$ 
 be a real valued function on ${\mathbb R}^n$, that satisfy the growth condition
 \begin{equation}\label{eq.growth}
  \lim_{|x|\rightarrow\infty} |F(x)|\exp\left(-a |x|^2\right)=0 \qquad \forall\,a>0.
 \end{equation}
 Then
 \begin{equation}
 {\mathbb E} XF(Y_1,\ldots,Y_n) = 
 \sum_{i=1}^n {\mathbb E} XY_i \; {\mathbb E} \partial_iF(Y_1,\ldots,Y_n).
\end{equation}
\end{lemma}

\noindent Involving the Gaussian Integration by Parts formula, we can further elaborate 
proposition \ref{cor-main1} in order to prove theorem \ref{thm.Stab-Log-Sob}.

More precisely, let 
${\cal G}_k$, be the class all the functions in ${\mathbb R}^k$, such that their first 
derivatives satisfy the growth condition \eqref{eq.growth}. Then for any $f\in{\cal G}_k$, 
lemma \ref{IbP} implies that
\begin{eqnarray*}
 {\mathbb E}\langle X , \nabla f(X) \rangle 
 &=& \sum_{i=1}^k {\mathbb E}X_i\partial_if(X)  \\
 &=& \sum_{i=1}^k\sum_{j=1}^k {\mathbb E}X_iX_j\,{\mathbb E}\partial_{ij}f(X)
 = {\mathbb E}\;{\rm tr}\big(TH_f(X)\big).
\end{eqnarray*}
where, $H_f(x)$ stands for the Hessian matrix of $f$ at $x\in{\mathbb R}^k$. In the special 
case where $X\sim N(0,I_k)$, we have proved the following

\begin{corollary}\label{cor-main2}
 Let $k\in{\mathbb N}$, and $X$ be a standard Gaussian vector in ${\mathbb R}^k$. Then
 \begin{itemize}
  \item [(i)]  for every log-concave function $f\in{\cal G}_k$ we have that
               \begin{equation}\label{eq.cor-main-2}
                {\rm Ent}_X(f) \geq 
                \frac12 {\mathbb E} \Delta f(X),
               \end{equation}
  \item [(ii)] for every log-convex function $f\in{\cal G}_k$ we have that
               \begin{equation}\label{eq.cor-main-2b}
                {\rm Ent}_X(f) \leq 
                \frac12 {\mathbb E} \Delta f(X).
               \end{equation}
 \end{itemize}
\end{corollary}

\begin{proof}[Proof of Theorem \ref{thm.Stab-Log-Sob}]
Let $f\in {\cal L}^{2,1}(\gamma_k)$, and without loss of of generality we may also assume 
that ${\mathbb E}f^2(X)=1$. Suppose first that $f$ has a bounded support. Then $f^2\in{\cal G}_k$,
and so Corollary \ref{cor-main2}, after an application of the chain rule 
$\frac12 \Delta f^2 = |\nabla f|^2 + f\Delta f$, gives that 
\begin{equation}\label{eq.stab-L-S}
 {\mathbb E}|\nabla f(X)|^2 + {\mathbb E}f(X)\Delta f(X) 
 \leq {\rm Ent}_X(f^2) \leq 2\,{\mathbb E} |\nabla f(X)|^2
\end{equation}

%
%
Finally, for $f=e^{-v}$, where $v: supp(f)\rightarrow{\mathbb R}$ is a convex 
function, and by another application of the chain rule:
\[
 f\Delta f = f^2  |\nabla v|^2 - f^2 \Delta v = |\nabla f|^2 - f^2 \Delta v,
\]
we get that 
\begin{equation} \label{eq.chainRule}
 {\mathbb E}f(X)\Delta f(X) = {\mathbb E} |\nabla f(X)|^2 - {\mathbb E}f(X)^2 \Delta v(X).
\end{equation}
Equation \eqref{eq.stab-L-S} combined with \eqref{eq.chainRule}, proves theorem \ref{thm.Stab-Log-Sob}
in this case.

\medskip

In order to drop the assumption of the  bounded support, we proceed with a standard 
approximation argument. We consider the functions $f_n:=f\,{\bf 1}_{nB_2^k}$, where 
${\bf 1}_{nB_2^k}$ is the indicator function of the Euclidean Ball in ${\mathbb R}^k$ 
with radius $n\in{\mathbb N}$. Then, every $f_n$ has bounded support and we also have that 
$0 \leq f_n \nearrow f$, $0\leq  |\nabla f_n|^2 \nearrow |\nabla f|^2$, and 
$0 \leq f_n^2\Delta v_n \nearrow f^2\Delta v$. Thus by the monotone convergence theorem 
we have
\begin{equation}\label{eq.mct_1}
 {\mathbb E} |\nabla f_n(X)|^2 \longrightarrow {\mathbb E} |\nabla f(X)|^2 < \infty
\end{equation}
and
\begin{equation}\label{eq.mct_2}
 {\mathbb E} f_n(X)^2\Delta v_n(X) \longrightarrow {\mathbb E} f(X)^2 \Delta v(X)
\end{equation}

Moreover, $f_n^2 \log f_n^2 \rightarrow f^2 \log f^2$ and $|f_n^2 \log f_n^2| \leq |f^2 \log f^2|$, 
for every $n\in {\mathbb N}$ (where we have taken that $0\log 0 =0$). By Gross' inequality 
$|f^2 \log f^2|\in L^1(\gamma_k)$, and so after applying the Lebesgue's dominated convergence theorem 
we also get
\begin{equation}\label{eq.dct}
 {\rm Ent}_X(f_n^2) \longrightarrow {\rm Ent}_X(f^2). 
\end{equation}
 Since equation \eqref{eq.quan-stab-L-S} holds true for every $f_n$, we pass to the limit using 
\eqref{eq.mct_1}, \eqref{eq.mct_2} and \eqref{eq.dct}, and we get that \eqref{eq.quan-stab-L-S} 
is also true for $f$. The proof is complete.
\end{proof}

\bigskip

\footnotesize

\bigskip

\noindent Nikos Dafnis  \\
Department of Mathematics  \\
Technion - Israel Institute of Technology \\
Haifa 32000, Israel. \\
nikdafnis@gmail.com         

\medskip

\noindent Grigoris Paouris \\
Department of Mathematics \\
Texas A\&M University \\
College Station, TX 77843, USA. \\
grigorios.paouris@gmail.com

\end{document}